\documentclass[12pt]{amsart}

\usepackage{epsfig,verbatim}
\usepackage{amsmath,amsfonts,amsthm,amsopn,cite,mathrsfs,amssymb}
\usepackage{subfigure,color,enumitem}

\newcommand{\tfa}{time-frequency analysis}

\newcommand{\ft}{Fourier transform}
\newcommand{\stft}{short-time Fourier transform}

\newcommand{\tf}{time-frequency}

\newcommand{\fif}{if and only if}
\newcommand{\tfs}{time-frequency shift}

\newcommand{\modsp}{modulation space}

\newcommand{\rep}{representation}

\newtheorem{tm}{Theorem}[section]    

\newtheorem{prop}[tm]{Proposition}

\theoremstyle{definition}

\newtheorem{rem}{Remark}[section]

\newcommand{\beqa}{\begin{eqnarray*}}
\newcommand{\eeqa}{\end{eqnarray*}}

\newcommand{\field}[1]{\mathbb{#1}}
\newcommand{\bR}{\field{R}}        
\newcommand{\bN}{\field{N}}        
\newcommand{\bZ}{\field{Z}}        
\newcommand{\bC}{\field{C}}        
\newcommand{\bH}{\field{H}}        %

\def\cS{\mathcal{ S}}

\def\cH{\mathcal{ H}}

\def\cG{\mathcal{ G}}

\def\cU{\mathcal{ U}}

\def\cN{\mathcal{ N}}

\def\rd{\bR^d}
\def\cd{\bC^d}
\def\rdd{{\bR^{2d}}}

\def\lrd{L^2(\rd)}

\def\intrd{\int_{\rd}}
\def\intrdd{\int_{\rdd}}

\def\<{\left<}
\def\>{\right>}

\def\inv{^{-1}}

\def\mv1{M_v^1}

\def\Mmpq{M_m^{p,q}}

\newcommand{\co}{Co _{\pi}}
\newcommand{\colam}{Co _{\pi _\lambda}}
\newcommand{\colpm}{\co L^p_m}
\newcommand{\colp}{\co L^p }
\newcommand{\lpm}{\ell ^p_m}
\newcommand{\id}{\mathrm{I}}
\begin{document}
\begin{abstract}
We study function spaces that are related to square-integrable,
irreducible,  unitary representations of several low-dimensional  nilpotent Lie groups. These are new examples of
coorbit theory and yield new families of function spaces on $\rd
$. The concrete realization of the representation suggests that these
function spaces are useful for  generalized time-frequency analysis
or phase-space analysis.  
\end{abstract}

\title[Coorbit Spaces for   Nilpotent Groups]{New Function Spaces
  Associated to  Representations of  Nilpotent Lie Groups and
  Generalized Time-Frequency Analysis}
\author{Karlheinz Gr\"ochenig}
\address{Faculty of Mathematics \\
University of Vienna \\
Oskar-Morgenstern-Platz 1 \\
A-1090 Vienna, Austria}
\email{karlheinz.groechenig@univie.ac.at}
\subjclass[2010]{42B35,22E25,46E35}
\date{}
\keywords{nilpotent Lie group, square-integrable representation modulo
center, coorbit space, modulation space, time-frequency analysis,
chirp, frame}
\thanks{K.\ G.\ was
  supported in part by the  project P31887-N32  of the
Austrian Science Fund (FWF)}
\maketitle

\section{Introduction}

The theory of coorbit spaces offers a systematic construction of
Banach spaces attached to a square-integrable representation of a
locally compact group. Roughly speaking, to any square-integrable, irreducible, unitary
representation $\pi $ of a locally compact group on a Hilbert space
$\cH $ one can associate a family
of Banach spaces by imposing a norm on the  representation
coefficients $x\in G \to \langle f, \pi (x) g\rangle $ for fixed
non-zero $g\in \cH$. For instance, the coorbit space $\colp (G)$ for $1\leq p\leq
2$ is   defined  by
the norm  $\|f\|_{\colp} = \Big(\int _G  |\langle f, \pi (x) g\rangle
|^p \, dx \Big)^{1/p} $.
Many important families of function spaces 
in analysis can be represented as coorbit spaces, among them are the
Besov-Triebel-Lizorkin spaces and  the \modsp s on $\rd $ or the
Bergman spaces on the unit ball in $\cd$.
The main theorem of
coorbit space theory provides atomic decompositions, sampling
theorems, and frames for all these spaces\cite{fg88lund,fg89jfa,gro91}.

In this paper we study  coorbit spaces for representations of
low-dimensional nilpotent Lie groups.  Our emphasis is on the concrete
realizations of these spaces (rather than abstraction) and  on the proof
that these families of  spaces are different from each other. 

Our motivation is twofold: (i) We explore  generalizations of \modsp s
and  \tfa . For this purpose we  propose coorbit spaces with respect
to a square-integrable representation of a nilpotent Lie group modulo
the center, as these function spaces  come automatically with an
underlying  configuration space and
a phase 
space. Phase-space analysis in
quantum mechanics and   \tfa\ in engineering are both  based on the Heisenberg group, and the
associated function spaces are the \modsp s. These arise naturally as
the coorbit spaces of the Schr\"odinger representation  and are now standard  for the
treatment of pseudodifferential operators~\cite{Sjo95,gro06},
the Schr\"odinger equation~~\cite{BGOR07,WH07},  evolution
equations~\cite{CNR15},  uncertainty principles~\cite{gro96} etc.
Except for the  isolated example of the Dynin-Folland group
in~\cite{FRR19} no  other
nilpotent groups have  been considered so far in coorbit theory. 
Yet every square-integrable representation (modulo the center) of a
nilpotent group comes automatically with a phase space and two sets of
variables that play the role of position and  momentum. Therefore
function spaces associated to such a representation 
offer themselves as tools for a generalized phase-space 
analysis.

By contrast, coorbit
space theory with respect to other types of groups have received
considerable attention and are well investigated. The coorbit theory
of certain semidirect products of $\rd 
$ with dilation groups has been intensely studied~\cite{BT96,CFT16,DKS09,fuehr15,FV20,ADMV17}. These
theories should be considered as contemporary 
group-theoretic investigations of  wavelet theory and carry a rather
different flavor and applications. Coorbit spaces with respect to
several semisimple Lie groups can be identified with Bergman spaces on certain
domains~\cite{CGO17,CO19} and coorbit space theory yields the most
general atomic decompositions known so far. 

Nilpotent groups have been used far less frequently. For
instance, the goal of ~\cite{FRR19} is a  formulation of  \tfa\ on the
Heisenberg group itself in place of $\rd $.  In this case the phase
space is  the
Dynin-Folland group and the associated function spaces are coorbit
spaces. There are several abstract  attempts to use a nilpotent Lie group $G$ or its Lie
algebra $\mathfrak{g}$ as a configuration spaces and then construct a
suitable phase space, for instance $G\times \hat{G}$ or $\mathfrak{g}
\times \mathfrak{g}^*$. The goal is then to associate generalized \modsp s to these
objects and  to develop a
pseudodifferential calculus for operators acting on $G$. For a sample of this abstract approach
see~\cite{Bel11,Eng17,FR16,Kis19,MP10,Man19,MR19}. 

Our point of view is  different: we start with an irreducible unitary representation of a
nilpotent Lie group that is square-integrable modulo the center and
automatically obtain a natural phase space.

(ii) Our second motivation comes from a question in ~\cite{FRR19}:
``precisely how can we prove the  distinctness'' [of different
families of function spaces]? Likewise, ~\cite{Man19} asks ``to invest effort \dots in
concreteness'' and ``to compare \modsp s with other function
spaces''. In this regard we offer a modest contribution and show that
the coorbit spaces with respect to three nilpotent Lie groups of
dimensions $5$ and $6$ indeed lead to different families of function
spaces on $\bR ^2$. In addition, we will give a short,  different,
self-contained  proof for the main result of \cite{FRR19}. 

Our  idea is to use the different invariance
properties of coorbit spaces with respect to non-isomorphic
groups. This idea leads to a simple, alternative proof of 
\cite[Thm.~7.6]{FRR19}. The  main technical tool is the precise
understanding of the multiplication operator $f(t) \to e^{-i\pi t^2}
f(t)$ on \modsp s, or equivalently, the Schr\"odinger evolution on
\modsp s. 

The investigation of invariance properties of function spaces  suggests that the coorbit spaces with
respect to two different representations $\pi _1$ and $\pi _2$  on a
Hilbert space $\cH $ are
 equal only when the groups $G_1/\mathrm{ker}\, \pi_1$ and
$G_2/\mathrm{ker} \,  \pi_2$ are isomorphic.  In this paper, we only treat
explicit representations of several low-dimensional
nilpotent Lie groups and defer the abstract question to a subsequent investigation.

The paper is organized as follows. In Section~2 we set up the
definition of abstract coorbit spaces and list their main
properties. As a model, we recall  how the  standard \modsp s
fit into the scheme of coorbit theory. We then explain where  to find the phase space in
every square-integrable irreducible unitary representation of a
nilpotent Lie group.  In Section~3 we walk through a list of concrete examples
of coorbit spaces based on representations of nilpotent Lie groups.
Our  main insight is the discovery of  several low-dimensional nilpotent 
Lie groups that yield a  class of coorbit spaces that are different
from the \modsp s. In Section~4 we collect several facts about atomic
decompositions of coorbit spaces and coherent frames and emphasize  the case of nilpotent groups.

Throughout we will use the notation  $f\asymp g$ to express an
equivalence $C\inv f(x) \leq g(x) \leq Cf(x), x\in X,$ with some
constant $C>0$ independent of the relevant parameters.




\section{ Coorbit Spaces}
\label{general}
Let $G$ be a simply connected nilpotent Lie group with Haar measure
$dx$ and with  center $Z \subseteq G$. Let $\pi : G \to \cU (\cH )$ be an irreducible, unitary \rep\ of
$G$ on a Hilbert space $\cH $ that is square-integrable modulo its center, in brief $\pi \in
SI(G/Z)$.  This means that there
exists a constant $d_\pi $, the formal dimension of $\pi $, such that
\begin{equation}
  \label{eq:2}
  \int _{ G/Z} \langle f_1, \pi (\dot{x})g_1 \rangle \,
  \overline{\langle f_2, \pi (\dot{x}) g_2 \rangle} \, d\dot x =
  d_\pi \inv \langle f_1, f_2 \rangle \overline{\langle g_1,g_2 \rangle}
\end{equation}
for all $f_1, f_2, g_1,g_2 \in \cH $. Since $\pi |_Z = \chi(z)
\mathrm{I}_\cH $ is a multiple of the identity with $\chi \in
\widehat{Z}$, the map  $x \in G \mapsto  \langle f_1, \pi (zx)g_1 \rangle \,
  \overline{\langle f_2, \pi (zx) g_2 \rangle} $ is independent of
  $z\in Z$ and the integrand in \eqref{eq:2} is indeed a function on
  $G/Z$. 

  For fixed non-zero $g\in \cH $ the representation coefficient
  \begin{equation}
    \label{eq:n99}
     V^\pi _gf(x) = \langle f, \pi (x) g\rangle \qquad x\in G
  \end{equation}
is a transform that maps elements  $f\in \cH $ to functions on
$G$. Depending on the group and the application, this mapping occurs
under various names, such as  coherent state transform, generalized wavelet
transform, or \stft . 

To extend the domain of the map $V_g^\pi$, we restrict  $g$ to a space of
test functions. For simplicity, 
we choose the space $\cH ^\infty _\pi $ of  $C^\infty$-vectors  of
$\pi $, where $g\in \cH ^\infty _\pi $  means that $x \in G \mapsto
\pi(x)g$ is $C^\infty$. For $g\in \cH ^\infty _\pi $ the
representation coefficient $V_g^\pi f$ is well-defined for
``distributions'' in $(\cH ^\infty _\pi )^*$. 
Other   choices of test function spaces  are discussed in
\cite{fg88lund,fg89jfa}.

\subsection{Coorbit spaces}

  We first  discuss the class of function and
distribution spaces associated to a \rep\ in $SI(G/Z)$. Fix a non-zero
vector  $g\in \cH ^\infty _\pi$. 
    Let    
$m: G/Z \to (0,\infty) $ be a weight functions of  polynomial growth
on $G/Z$,  and let $1\leq p\leq
\infty $. In the following we always assume that $m$ satisfies the
condition
\begin{equation}
  \label{eq:h1}
v(x) :=   \sup _{y\in G}   \{ \frac{m(xy)}{m(y)}, 
\frac{m(yx)}{m(y)}\} < \infty  \qquad \text{ for all } x \in G/Z \, ,
\end{equation}
and that $v$ grows polynomially on $G/Z$. We say that $m$ is moderate
with respect to $v$ or $v$-moderate on $G/Z$. Note that \eqref{eq:h1}
implies that $m(xy) \leq v(x) m(y)$ and $v(xy)\leq v(x) v(y)$. 

 For $p<\infty $ the coorbit space $\co L^p_m(G/Z)$  is the
completion  of the subspace of the space $\cH ^\infty _\pi $ of $C^\infty$-vectors with respect to the norm
\begin{equation}
  \label{eq:4}
  \|f\|_{\colpm} = \Big(\int _{G/Z} |\langle f, \pi (\dot x) g \rangle |^p
  \, m(x)^p \, d\dot x \Big)^{1/p} = \|V_g^\pi f\|_{L^p_m(G/Z)}  \, .
\end{equation}
For $p=\infty $ one takes a weak closure as suggested in
~\cite{FoG05,BG17a}.
A more general definition starts with a solid function
space $Y$ on $G/Z$ satisfying certain  natural conditions,  and then
one   defines a
norm on functions via pull-back, i.e.,  $\co Y$ is the completion
of $C^\infty$-vectors  in $\cH $ with respect to the norm 
\begin{equation}
  \label{eq:n5}
  \|f\|_{\co Y} = \| V_g^\pi f\|_Y \, .
\end{equation}
 We refer to ~\cite{fg89jfa}
for the precise conditions and details.  In this paper we will use
only weighted $L^p$-spaces.

For nilpotent groups the Hilbert space of $\pi $ can always be realized as $\lrd $, where
$\rd $ is to be interpreted as a homogeneous space $G/M$ for a
polarization $M$. Then $g$ can be taken from the Schwartz class $\cS
(\rd )$ by~\cite[4.1.2]{CG90}. In this realization,  $\colpm (G/Z)$ is the subspace
of tempered distributions $f \in \cS ' (\rd)$ such that the
representation coefficient  $\dot{x} \to \langle f, \pi (\dot
x)g\rangle $ belongs to $L^p_m(G/Z)$. 

We summarize the main properties of a coorbit space from
\cite{fg89jfa}.
\begin{prop} \label{props}
  Assume that $1\leq p\leq \infty $ and that $m$ is $v$-moderate.
  
  (i) \emph{Invariance properties:} The coorbit space $\colpm (G/Z)$,
  $1\leq p \leq \infty $    is a Banach space with the norm 
\eqref{eq:4}. It is invariant with respect to the \rep\ $\pi $,
precisely, if $f\in \colpm (G/Z)$ and $y\in G$, then $\pi (y) f\in \colpm (G/Z) $
and $\|\pi (y) f \|_{\colpm} \leq v(y) \|f\|_{\colpm }$. In
particular, $\pi (y)$ is an isometry on $\co L^p(G/Z)$. \\

(ii)  If $h\in \co L^1_v$,  
 then $\|V_h^\pi f \|_{L^p_m}$ is an
equivalent norm on $\colpm (G/Z) $.

(iii) \emph{Duality:} The dual space of $\colpm (G/Z)$ for $1\leq p < \infty $ is $\co
L^{p'}_{1/m}(G/Z)$ with respect to the duality $\langle f, h\rangle =
\int _{G/Z} V_g^\pi f(\dot{x}) \overline{ V_g^\pi  h(\dot{x}) }\, d\dot{x}$.

(iv) As a Banach space, $\colpm (G/Z)$ is isomorphic to $\ell ^p$.

(v) For $p=2$ we have $\co L^2(G/Z) = \cH $. If $m\geq 1$ on $G/Z$ and
$1\leq p\leq 2$, then $\colpm (G/Z) \subseteq \cH $. 
\end{prop}

The results in~\cite{fg89jfa} are formulated for a square-integrable
representation of a locally compact group $G$.  For a simply connected nilpotent Lie group we
have to consider  square-integrability modulo the center and thus
subsequently we work with  $G/Z$ instead of $G$. Alternatively, one
could work with projective representations of $G/Z$ as mentioned
in~\cite{fg88lund,Chr96}. This is a matter of convenience and taste, and  we prefer to
work with representation of a given group $G$. 

\subsection{Modulation spaces} Most classical function spaces, such as
Sobolev spaces and Besov spaces on $\rd $ can be interpreted as
coorbit spaces with respect to a solvable group of affine
transformations. The most important coorbit spaces attached to a
nilpotent group are the \modsp s originally introduced by H.\
Feichtinger~\cite{fei83}.

We briefly recall their definition. In the following we  write
\begin{equation}
  \label{eq:n4}
  T_xf(t) = f(t-x) \qquad \text{ and } \qquad  M_\xi f(t) = 
 e^{2\pi i \xi \cdot t} \qquad x,\xi ,t\in \rd  \, ,
\end{equation}
 for the operators of translation and  modulation. 
The \stft\ of a
function $f$ with respect to a fixed non-zero window $g\in \cS (\rd )$ is  given
by 
\begin{equation}
  \label{eq:h2}
S_gf(x,\xi ) = \langle f, M_\xi T_x  g \rangle = \intrd f(t)
\overline{g(t-x)} e^{-2\pi i t\cdot \xi  } \, dt \, .  
\end{equation}
In signal processing  the variable $x$ is interpreted as ``time'' and
$\xi $ as ``frequency'', and  the pair $(x,\xi)$ is a point in
time-frequency space $\rdd $.   In the language of
physics, in particular in quantum mechanics,  $x$ corresponds to  the
position and $\xi $ to the
momentum, and $(x,\xi )$ is a point in phase space $\rdd $. The \stft\ is an  important tool to study the
simultaneous \tf\ distribution (or phase-space distribution)  of a
signal $f$. 

The \modsp s are defined by imposing a norm on the \stft\
$S_gf$. Every solid  norm 
$\| \cdot \|_Y$  for functions on phase space $\rdd $ induces a norm
on functions or distributions on $\rd $ by the rule
$$
\|f\|_{M(Y)} = \|S_gf\|_Y \, .
$$
The resulting space $M(Y)$ is the \modsp\ corresponding to $Y$. The
standard choice for $Y$ is a weighted $L^p$-space. Let $m\geq 0$ be a
weight function on $\rdd $, then the \modsp\ $M^p_m(\rd )$ is defined by the
norm
\begin{equation}
  \label{eq:n4b}
\|f\|_{M^p_{m}} = \Big( \intrdd |S _gf(x,\xi)|^{p} m(x,\xi)^p \, dxd\xi
\Big)^{1/p} \, .
\end{equation}
The needs of \tfa\ often require mixed $L^p$-spaces, so one often
looks at the \modsp\ $\Mmpq (\rd )$ defined by the norm
$$
\|f\|_{\Mmpq} = \Big( \Big(\intrd |S _gf(x,\xi)|^{p} m(x,\xi)^p \, dx
\Big)^{q/p} d\xi
\Big)^{1/p} \, .
$$
Thus the $M^p_m$ and $\Mmpq$-norms  quantify the phase-space content of $f$. 

The theory of  \modsp s  forms a branch of \tfa\ 
and hardly needs any further advertisement in this context. Modulation spaces  have
become indispensible for \tfa\ and phase-space analysis and  for the study of pseudodifferential
operators.     The
original paper is  Feichtinger's~\cite{fei83}, for the
history see~\cite{feiSTSIP}, a 
detailed exposition  is contained in~\cite{book} and the forthcoming
books~\cite{BO21,CR21}.  

In our context \modsp s serve as the prototype of function spaces
associated to a representation of a nilpotent Lie group. 

\subsection{Chirps} Our main tool for distinguishing various coorbit
spaces is the behavior of a certain multiplication operator on these
spaces. Let $C = C^T$ be a real-valued symmetric $d\times d$-matrix
and
\begin{equation}
  \label{eq:n9}
  \cN _C f(t) = e^{-i\pi Ct\cdot t} f(t)
\end{equation}
be the corresponding multiplication operator. In engineering
terminology the multiplier $e^{-i\pi Ct\cdot t}$ is called a
``chirp'', in quantum mechanics and PDE $\cN _C$ is an ingredient in
the solution formula for the free Schr\"odinger equation. Given $C$,
we define $D= (4\id + C^2)\inv $ and the $2d\times2d$-matrix $ \Delta =
\Big(\begin{smallmatrix}
  2D & - DC \\ -DC & \emph{I}-2D 
\end{smallmatrix} \Big)$ acting on $z\in \bR ^{2d}$.

We will need
and use a precise estimate for action of $\cN _C$ on \modsp s.

\begin{prop} \label{chirp}
  Let $\phi (t) = e^{-\pi t\cdot t} $ be the standard Gaussian on $\rd
  $ and   $1\leq p\leq 2$.  Then   the modulus of the \stft\ is
\begin{equation}
  \label{eq:app2}
  |S_\phi (\cN _C\phi )(x,\xi )| = \det (4\id + C^2)^{-1/4} e^{-\pi
    \Delta (\xi, x )^T\cdot (\xi , x )^T} \,  ,
\end{equation}
and its $p$-norm is therefore
\begin{align}
  \|\cN _C\phi \|_{M^p(\rd )} = \det (4\id + C^2)^{\frac{1}{2p}-\frac{1}{4}}
  \, .  \label{eq:app3}
\end{align}
\end{prop}
For  $d=1$,  $C=u$, and $\cN _u = e^{-\pi i ut^2}$ we obtain
\begin{equation}
  \label{eq:app0}
  |S_\phi (\cN _u \phi ) (x,\xi )| = (4 + u^2 ) ^{-1/4} \, e^{-\pi
    |x|^2/2} \,  e^{-2\pi (\xi + ux/2)^2/(4+u^2)} 
\end{equation}
\begin{equation}
  \label{eq:m1}
\|S_\phi (\cN _u \phi )\|_{L^p(\bR ^{2})} = c_p
(4+u^2)^{\frac{1}{2p}-\frac{1}{4}} \asymp |u|^{\frac{1}{p}-\frac{1}{2}} \, .
\end{equation}
 For $d=2$, $uC=\Big(
\begin{smallmatrix}
  0 & u/2 \\ u/2 &0
\end{smallmatrix}\Big)$, and  $\cN _{uC} = e^{-\pi i uCt\cdot t}$ we
have we have  $(uC)^2 =\tfrac{u^2}{4} \id $  and therefore
$\det (4\id + u^2/4 \id ) = (4+u^2/4)^2$ and thus
\begin{equation}
  \label{eq:m2}
\|S_\phi (\cN _{uC}\phi )\|_{L^p(\bR ^4)} =
(4+u^2/4)^{\frac{1}{p}-\frac{1}{2}} \asymp |u|^{\frac{2}{p}-1 }\, .  
\end{equation}

The proof of Proposition~\ref{chirp} consists of the manipulation of
Gaussian integrals as in~\cite{folland89}. In the context of \modsp s the
operator 
norm of $\cN _C$ has been calculated in~\cite[Lemma~5.3]{CN08a},
among others. For convenience we offer a short proof in our notation
in the appendix.


\subsection{Where is phase space?} \label{psa} The construction of coorbit spaces
works for arbitrary integrable, irreducible, unitary representions of
a  locally compact group. To understand why the representations of a 
nilpotent Lie group yield a form of \tfa\ or phase-space analysis, we
need to look at the general form of the irreducible unitary
representations of nilpotent Lie groups, e.g., in ~\cite{CG90}.

By Kirillov's theory every irreducible representation of a (simply
connected, connected) nilpotent Lie group is induced from a character
of a particular subgroup $M$ of $G$. Fix a functional $\ell \in
\mathfrak{g}^*$ and a maximal Lie subalgebra $\mathfrak{m} \subseteq \mathfrak{g}$, a
so-called polarization, such
that $\ell ([X,Y]) = 0 $ for all $X,Y\in \mathfrak{m}$. Then $\ell $
defines a character $\chi $ on the subgroup $M = \exp ( \mathfrak{m})$
by $\chi (m) = e^{2\pi i \langle \ell , \log m \rangle }$ for $m\in
M$.    The representation $\pi = \pi _\ell $ is the representation
induced from $(\chi , M)$ to $G$. Every irreducible unitary
representation of $G$ can be obtained in this way from some $\ell \in \mathfrak{g}^*$. 

 This induced representation can be
realized explicitly on the representation  space 
$L^2(M\backslash G)$ with a $G$-invariant measure on $M\backslash G$. If  $p :G \to
M\backslash G:=H$ denotes the projection  and $\sigma :M\backslash G \to
G$  a (continuous) section, then every element in $G$ can be written
uniquely as  
$g = mh $ 
with $ m = g \sigma (q(g))\inv \in M$ and $h=\sigma (g) \in H$. The
induced representation is then
\begin{equation}
  \label{eq:n8}
(\pi(g)f) (q(h')) = e^{2  \pi i \langle l, \log\big(h' m h'^{-1} (h'
h\sigma (h'h)\inv) \big) \rangle} \, f\bigl(q(h'\hspace{2pt} h)\bigr) \, .  
\end{equation}                
If $\pi $ is square-integrable modulo the center, then the radical of
$\ell $ coincides with the center~\cite{CG90} and  
$\mathrm{dim}\, \mathfrak{g}/\mathfrak{z}=  2d$ is even.  Let $r=
\mathrm{dim}\, \mathfrak{z}$  and $\mathrm{dim}\,
\mathfrak{m}= r + d$, so that $\mathrm{dim}\,
\mathfrak{g}/\mathfrak{m}=  d$.   Now choose a (strong) Malcev basis $\{Z_1, \dots , Z_r, Y_1, \dots , Y_d, X_1, \dots , X_d
\}$ of
$\mathfrak{g}$ passing through $\mathfrak{z}$ and $\mathfrak{m}$,
 such that $\mathfrak{m} = $ $ \mathrm{span} \, \{Z_1, \dots , Z_r, Y_1, \dots ,
Y_d\} $. Then every $m\in M$ can be written as $m= e^{z_1Z_1} \dots
e^{z_rZ_r} e^{y_1Y_1} \dots   e^{y_dY_d}$, and a distinguished
section $\sigma : M\backslash G \to G$ is $\sigma (Mh) = e^{t_1X_1} \dots
e^{t_d X_d }$. With this choice of coordinates, the multiplicative
term in \eqref{eq:n8} is $e^{2  \pi i \langle l, \log\bigl(h' m
  h'^{-1} \bigr) \rangle}  = e^{2\pi i P(y,t)}$ with $P$ a polynomial
in the coordinates $y$ of $m$ and $t$ of $h'\in M\backslash G$. Thus this part of
$\pi $ can be interpreted as a generalized modulation, with the
polynomial expression $e^{2\pi i
  P(y,t)}$ replacing the linear expression $M_\xi = e^{2\pi i \xi
  \cdot t}$. The  group action $Mh' \mapsto
Mh'h$ is a  generalized translation. The representation is therefore
roughly of the form
\begin{equation}
  \label{eq:1}
g(t) \to  \pi (x,\xi
) g(t) = e^{2\pi i   P(t,\xi)} g(t  x)   \, ,
\end{equation}
 where $P$ is a polynomial in
$t,\xi$ and $t \to t x$ is a group action. Thus $\pi $ splits into a
generalized modulation and a generalized translation. 
In this sense the induced
representation can be viewed as a generalized \tfs\ on the
configuration  space
$M\backslash G$. 

Usually one cannot neglect the factor $e^{2  \pi i \langle l, \log\big( h'
h\sigma (h'h)\inv \big) \rangle} $ depending only on $H=M\backslash
G$. In our concrete examples it disappears. Precisely, whenever $G$
splits as a semidirect product $G = M \ltimes H$ for a subgroup $H$,
then $h'
h\sigma (h'h) = e \in M$ and this factor is absent.

Omitting the central
coordinates $z_j$, we interpret  the variables $y = (y_1, \dots ,
y_d)\in \rd $  parametrizing $m\in M$ as ``frequency/momentum'' and the variables
$(x_1, \dots, x_d) \in \rd $ parametrizing $h= e^{x_1X_1} \dots
e^{x_dX_d} \in M\backslash G$  as
``time/position''.   

In this  analogy  the associated representation coefficient
$$
V^\pi _g f(\dot x) = \langle f,\pi (\dot x) g\rangle \, 
$$
is a version of the \stft\ measuring a new kind of phase space
concentration on $M\backslash G$. 

\section{Concrete examples} \label{examples}
We give some concrete  examples of coorbit spaces with respect to
nilpotent groups, some already known, some new. 

For the low dimensional examples we use the classification of
Nielsen~\cite{Niel83}. This source lists all nilpotent Lie groups of
dimension $\leq 6$ with their Lie algebras, the explicit group
multiplications, descriptions of the coadjoint orbits, and the
associated irreducible representations. By using the available
formulas, we can write the  representations and coorbit spaces without
requiring   Kirillov's theory. 

\subsection{The Heisenberg group and \modsp s }

Let $\bH _d= \bR ^{d} \times \rd  \times \bR $ with multiplication $(x,y,z)\cdot (u,v,w)
= (x+u,y+v, z+w + x\cdot v)$ and the representation $\pi _\lambda ,
\lambda \neq 0$, acting on $L^2(\bR ^d)$ by the operators 
\begin{equation}
  \label{eq:schr}
\pi _\lambda (x,y,z)f(t) =
e^{2\pi i \lambda z} e^{-2\pi i \lambda y\cdot t} f(t-x)   
\end{equation}
for $f\in \lrd $.
This
representation, the Schr\"odinger representation,  is square-integrable modulo the center $\{0\}\times
\{0\} \times \bR $.

 Omitting the center (or  setting $z=0$) and comparing to \eqref{eq:h2},  we see that the
representation coefficients of $\pi _\lambda $ are simply a scaled
version of the 
short-time \ft : 
$$
V_g^{\pi _\lambda} f(x,y,0) = \langle f, M_{-\lambda y} T_x g \rangle =
S_gf(x,-\lambda y) \, .
$$
Since by definition the coorbit space norm on $\colpm $ is 
\begin{align*}
\|f\|_{\colpm } ^p &= \intrdd |V^{\pi_\lambda} _gf(x,y)|^{p}
                     m(x,y)^{p} \, dxdy \\
  &=  \intrdd |S _gf(x,-\lambda y)|^{p} m(x,y)^{p} \, dxdy =  \lambda ^{-p}
\,  \|f\|_{M^p_{m_\lambda}}^p 
  \end{align*}
  with the weight $m_\lambda (x,y) =  m(x, -\lambda\inv y)$. 
The class of  coorbit spaces of the Heisenberg group is therefore
identical to the class   of  the  \modsp s.
This group-theoretic point of view on \modsp s is 
explained in detail in~\cite{fg92chui}.

\subsection{The group $G_{6,16}$}
This group is the six-dimensional group $G_{6,16} \simeq \bR ^6$
defined by the Lie brackets
$$
[X_6,X_5]=X_2, \quad [X_6,X_4]=X_1, \quad [X_5,X_3]=X_1 
$$
with group multiplication
\begin{align*}
x\cdot y &= (x_1, \dots , x_6)\cdot (y_1, \dots , y_6)\\ & = (x_1+y_1 +
x_5y_3 + x_6y_4, x_2+y_2+x_6y_5, x_3+y_3, x_4+ y_4, x_5+ y_5,  x_6+
y_6) \, .
\end{align*}
It possesses the two-parameter family of 
representations modulo the center $\bR ^2 \times \{ (0,0,0,0)\}$
acting on $L^2(\bR ^2)$ by the operators
$$
\pi _{\lambda ,\mu } (x_1, \dots , x_6) g(s,t) = \exp  2\pi i \Big(
\lambda (x_1-x_3 s -x_4 t)   +\mu (x_2-x_5x_6 +x_6s) \Big)
g(s-x_5,t-x_6) \, .
$$
For $\lambda \neq 0, \mu \in \bR $ the representation
$\pi_{\lambda,\mu}$ is square-integrable modulo the center. 
When using the absolute values $|V^{\pi _{\lambda ,\mu }} _gf|$, the
phase factor $e^{2\pi i (\lambda x_1 +\mu (x_2-x_5x_6))}$ 
disappears, and we
will therefore omit it in the definition of the transform  $V^\pi _g$.
In addition, we identify the quotient $G_{6,16}/Z$ with a convenient  section
$G_{6,16}/Z \to G_{6,16}$ and will write $\dot x = (0,0,x_3, \dots , x_6) \in
G_{6,16}/Z$. Then  we
can write the representation  with the help of  \tfs s as follows:
\begin{equation}
  \label{eq:l8}
  \pi _{\lambda , \mu }(\dot x) = M_{(-\lambda x_3 + \mu x_6, -\lambda
    x_4)} T_{(x_5,x_6)} \, .
\end{equation}
So the cleaned-up \rep\ coefficient is just a scaled version of the
\stft\
$$
V_g^{\pi _{\lambda , \mu }}(\dot x) = S_gf\big((x_5,x_6), (-\lambda x_3 + \mu x_6, -\lambda
x_4)\big) \, .
$$
The resulting coorbit space norm $\colpm $ is then
$$
\| f \|_{\colpm} ^p = \int _{\bR ^4} |V^{\pi _{\lambda ,\mu }}_g f
  (\dot x)| ^p \, m(\dot x )^{p} \, d\dot x  \, .
  $$
  Using the linear coordinate transform $T\dot{x}=y$,  $y_3 = -\lambda x_3 + \mu x_6,
  y_4 = -\lambda     x_4, y_5 = x_5, y_6=x_6$, we find that
$$\|f\|_{\colpm (G/Z) } ^p = \int _{G/Z} |S_gf(T\dot{x})|^p m(\dot{x})^p
\, d\dot x = c \int _{\bR ^4} |S_gf(y)|^p m(T\inv y) \, dy  \, . 
$$
Therefore $Co _{\pi_{\lambda,\mu}} L^p(G_{6,16}/Z) = M^p_{m\circ T\inv }(\bR ^2)$ is just a \modsp\
with a modified weight that  depends  on the parameters $\lambda ,\mu $ of
the representation $\pi_{\lambda,\mu}$. Thus the coorbit spaces
associated to a representation $\pi_{\lambda,\mu }$  of the group
$G_{6,16}$ all belong to the class of \modsp s, and  no new spaces
arise from this group. 

For the special case 
of a polynomial weight $m(\dot x) = (1+|\dot x|)^s, s\in \bR $, we
have  $m\circ T\inv  \asymp m$, and  we see that
  \eqref{eq:l8} is just an equivalent norm for the \modsp\ $M^p_m (\bR
  ^2)$ independent of the parameters of the representation $\pi
  _{\lambda,\mu}$.  

  In conclusion,  the group $G_{6,16}$ does not yield any new coorbit
  spaces, nor a new version of \tfa .  This seems intuitive, because
  the quotient $G_{6,16}/Z $ is the abelian group $\bR   ^4$, as is
  the quotient $\bH _2 / Z $.

  Likewise the groups studied in~\cite{Kis19,Eng17}
  for the construction of general coherent states lead to the standard
  \modsp s (because $G/Z \simeq \rdd $).

\subsection{The group $G_{5,3}$}
This is the (simply connected) nilpotent Lie   group generated by the
Lie algebra with the  Lie brackets
\begin{equation}
  \label{eq:l1}
  [X_5,X_4]=X_2 , \quad   [X_5,X_2]=X_1 , \quad   [X_4,X_3]=X_1  \, .
\end{equation}
For $x= (x_1,x_2, \dots , x_5), y= (y_1, \dots, y_5) \in \bR ^5$ the  group  multiplication
of $G_{5,3}$ is given by
\begin{equation}
  \label{eq:l2}
  x\cdot y = (x_1+y_1+x_4y_3+x_5y_2+\tfrac{1}{2} x_5^2 y_4,
  x_2+y_2+x_5y_4, x_3+y_3, x_4+y_4, x_5+y_5) \, .
\end{equation}
The center of $G_{5,3}$ is $\bR \times \{(0,0,0,0)\}$. The analysis of
all irreducible unitary representations yields a one-parameter family
of representations square-integrable modulo the center $\pi _\lambda
\in SI (G/Z), \lambda \neq 0 $.  For $x\in G_{5,3}$ and $(s,t) \in \bR ^2$,  $\pi _\lambda
$   acts on $g\in L^2(\bR ^2)$ as
follows:
\begin{equation}
  \label{eq:l3}
  \pi _\lambda (x)g(s,t) =  \exp \Big( 2\pi i \lambda (x_1 - x_3x_4 +
  x_4 s -x_2t +\tfrac{1}{2} x_4 t^2) \Big) \, g(s-x_3, t-x_5) \, .
\end{equation}
Let us briefly analyze this representation from the perspective of
\tfa . The operators $\pi _\lambda (x)$ act on $g$ by the translations
$T_{(x_3,x_5)}$ and the modulations $M_{(\lambda x_4,-\lambda x_2)}$. The additional
factor that makes this representation interesting and different from
the Schr\"odinger representation \eqref{eq:schr} is the multiplication by the chirp
$e^{\pi i \lambda  x_4 t^2}$. 
We see that $\pi _\lambda $ is indeed of the form $e^{2\pi i P(x;s,t)}
T_u$ as motivated in \eqref{eq:1} for some polynomial $P$ and some
translation by $u=(x_3,x_5)$.  

The polarization used for $\pi _\lambda $ is  $\mathfrak{m} = \bR -
\mathrm{span}\, \{ X_1,X_2,X_4\}$. As explained in Section~\ref{psa}, 
 the underlying phase space consists of the
``frequency'' variables $x_2,x_4$ and the ``time'' variables
$x_3,x_5$, and   
 the pairs $(x_2,x_5)$ and
$(x_3,x_4)$ are  conjugate variables, since $[X_5,X_2], [X_3,X_4] \in
\mathfrak{z}$. 

Omitting the phase factor $e^{2\pi i \lambda
  (x_1-x_3x_4)}$, which disappears  in  absolute values of $V_g^{\pi
  _\lambda }f$,  and choosing the section $\dot x = (0,x_2,
x_3,x_4,x_5)$, the
associated representation coefficient with respect to a fixed $g\in
L^2(\bR ^2)$ is given by
\begin{equation}
  \label{eq:l4}
  V_g^{\pi _\lambda }f(\dot x) = \int _{\bR ^2} f(s,t) \,
  \bar g(s-x_3, t-x_5)  \exp \Big( -2\pi i \lambda (x_4 s -x_2t
  +\tfrac{1}{2} x_4 t^2) \Big) \, dsdt \, .
\end{equation}
This formula certainly justifies the interpretation of   $ V_g^{\pi _\lambda }g(x_2,x_3,
x_4,x_5)$ as a generalized \stft . 


Let $m$ be a moderate,  polynomially growing weight on $\bR ^4$,
$1\leq p\leq \infty$,  and fix 
$g\in \cS (\bR ^2), g\neq 0$. The family of coorbit spaces
with respect to $\pi _\lambda $ is given by
\begin{equation}
  \label{eq:l5}
  \colam  L^p_m(G_{5,3}/Z) = \{ f\in \cS ' (\bR ^2): V_g^{\pi
    _\lambda } f \in L^p_m(G_{5,3}/Z) \} 
\end{equation}
with norm
\begin{equation}
  \label{eq:l6}
  \|f\|_{\colam  L^p_m}^p = \int _{\bR ^4} |V_g^{\pi
    _\lambda } f(x_2,x_3, x_4,x_5)|^p \, m(x_2,x_3, x_4,x_5)^p \,
  dx_2dx_3dx_4dx_5 \, .
\end{equation}
In analogy to the mixed \modsp s $M^{p,q}_m(\rd )$  one might also
consider  the spaces with norm
$$
  \|f\|_{\colam  L^{p,q}_m}^p = \int _{\bR ^2} \Big(\int _{\bR ^2} |V_g^{\pi
    _\lambda } f(x_2,x_3, x_4,x_5)|^p \, m(x_2,x_3, x_4,x_5)^p \,
  dx_3dx_5 \Big)^{q/p} \, dx_2 dx_4  \, .
$$
Thus  the associated
coorbit spaces bear some resemblance to the  \modsp s and can be rightly
considered generalized \modsp s. 
This  class of function spaces should be  well
suited for all questions concerning the operators $\pi _\lambda $. 

Notice that
$$
|V_g ^{\pi _\lambda } f(x_2, x_3, x_4,x_5)| = |V_g^{\pi _1}f(\lambda
x_2, x_3, \lambda x_4, x_5)| \, ,
$$
so that   $V_g^{\pi _\lambda }f \in L^p_m(G_{5,3}/Z)$ \fif\ $V_g^{\pi _1 }f
\in L^p_{m_\lambda}(G_{5,3}/Z)$ with $m_\lambda (x_2, x_3, x_4,x_5) =
m(\lambda\inv x_2, x_3, \lambda \inv x_4,x_5)$. Thus different
parameters $\lambda $ yield the same family of spaces, possibly with a
change of the weight. 
We may therefore assume without loss of generality
that $\lambda =1$ and write $\pi = \pi _1$ in the following.

Our main insight   is  that the coorbit spaces $\colpm
(G_{5,3}/Z)$  form a new class of function spaces on $\bR ^2$ and
differ from the standard \modsp s.

\begin{prop} \label{neqs}
  Let $p,q\in [1,\infty ], p\neq 2$. 
  Then
  $$
  \co L^p(G_{5,3}) \neq M^q  (\bR ^2) \, .
  $$
\end{prop}

\begin{proof}[First proof]
  Because of the structure of the representation we use tensor
  products $f(s,t) = f_1(s) f_2(t) = (f_1 \otimes f_2) (s,t)$. To
  measure the norm, we choose a Gaussian window $g(s,t) = e^{-\pi s^2}
  e^{-\pi t^2}= \phi \otimes \phi (s,t)$ with the Gaussian $\phi (s) =
  e^{-\pi s^2}$ for $s\in \bR $.

  Let
  $$\cN _uf(t) = e^{-i\pi ut^2} f(t) $$
  be the multiplication
  operator with the ``chirp'' $e^{-i\pi ut^2}$ acting on $L^2(\bR )$,  and $T_x, M_\xi $ be
  the \tfs s  on $L^2(\bR )$. Then the representation
  coefficients of a tensor product can be written as
  \begin{align*}
    \lefteqn{    V_g^\pi (f_1 \otimes f_2)(x_2, x_3, x_4, x_5) \ } \\
    &= \int _{\bR } f_1(s) \phi (s-x_3) e^{-2\pi i x_4s} \, ds 
      \int _{\bR } e^{-i\pi x_4t^2} \, f_2(t) \phi (t-x_5) e^{2\pi i
      x_2t} \, dt \\
    &= S_\phi f_1(x_3,x_4) \, S_\phi (\cN _{x_4} f_2)(x_5,-x_2) \,
      dt \,  .
  \end{align*}
Taking  the $p$-norm first with respect to $x_2,x_5$, we obtain
  \begin{align}
  \| V_g^\pi (f_1 \otimes f_2)\|_{L^p(G/Z)}^p  &= \int _{\bR ^2}
 \|S_\phi (\cN  _{x_4}f_2)\|_{L^p(\bR ^2)}^p |S_\phi f_1(x_3,x_4)|^p \,
        dx_3 dx_4 \notag  \\
    &=  \int _{\bR ^2} \|\cN  _{x_4}f_2\|_{M^p(\bR )}^p |S_\phi f_1(x_3,x_4)|^p \,
      dx_3 dx_4  \, .  \label{eq:h3}
  \end{align}
%

We now set   $v(x_3,x_4) = (1+|x_4|)^{\frac{1}{2p} -
  \frac{1}{4}}$ and  choose $f_2 = \phi $ and $f_1 \in M^p(\bR )
$ arbitrary. Since
\begin{equation}
  \label{eq:fst}
S_{\phi \otimes \phi } (f_1 \otimes
f_2)(x_1,x_2,\xi_1,\xi_2) = S_\phi f_1(x_1,\xi_1) S_\phi
f_2(x_2,\xi_2) \, ,  
\end{equation}
we
see that $f_1 \otimes \phi \in M^p(\bR ^2)$. 
On the other hand,   using Proposition~\ref{chirp} the \modsp\ norm
of  $\cN _{x_4} \phi $ 
is 
\begin{equation}
  \label{eq:a3}
  \|\cN _{x_4} \phi \|_{M^p(\bR )} \asymp  \|S_\phi (\cN _{x_4}\phi) \|_{L^p(\bR ^2)} \asymp
  (4+x_4^2)^{\frac{1}{2p} - \frac{1}{4}} \asymp
  |x_4|^{\frac{1}{p} - \frac{1}{2}}  \, .
\end{equation}
Continuing \eqref{eq:h3}, we find that
$$
\|f_1 \otimes \phi \|_{\colp } ^p \asymp \int _{\bR ^2} (1+x_4^2)^{1/2-p/4}
|S_\phi f_1 (x_3,x_4)|^p \, dx_3dx_4 \asymp \|f_1\|_{M^p_v}^p \, \|\phi
\|_{M^p} ^p \, .
$$
Thus $f_1\otimes \phi \in \colp (G_{5,3}/Z)$, 
\fif\ $f_1 \in
M^p_v(\bR )$. If $1\leq p<2$ and $p'= p/(p-1)<2$, then  $M^p_v$ is a proper subspace of $M^p$, e.g., by
~\cite[Cor.~12.3.5]{book}. Therefore  there exists $f_1 \in M^p \setminus M^p_v$, and
consequently, we have constructed an element 
$f_1 \otimes \phi \in M^p(\bR ^2)$, but $f_1 \otimes \phi  \not \in
\colp $.

So far, we have proved that $\colp \neq M^p$ for $1\leq p<2$. 
For $p>2$ we use the duality. By Proposition~\ref{props}(iii), $\colp  \simeq  (\co L^{p'})^* \neq (M^{p'})^*
\simeq M^p$.   

Finally, we argue  that $\colp \neq M^q$ for $q\neq p$. By~\cite[Prop.~9.3]{fg89mh}
$\colp $ is isomorphic to $\ell ^p$, whereas $M^q $ is  
isomorphic to $\ell ^q$, whence these spaces must be different. 
\end{proof}

%


The above argument does not work for the case $p=2$. This is clear,
because we always have $\co L^2(G/Z) = \cH = L^2(\bR ^2)$ by Proposition~\ref{props}(v). 

Next we  recast the proof  in a different  form
that is more  suitable for generalization. 

\begin{proof}[Second proof]

By  Proposition~\ref{props} the representation $\pi $
is an isometry on $\colp (G_{5,3}/Z)$, therefore  we have
\begin{equation}
  \label{eq:n9b}
\|\pi (0,0,x_4,0)f \|_{\colp } = \|f \|_{\colp } \, .  
\end{equation}
By contrast, since by \eqref{eq:l3}  $$\pi (0,0,x_4,0)(f_1 \otimes f_2)  = M_{x_4}f_1
\otimes  \cN _{x_4}f_2$$ and
the 
modulation is  an isometry on $M^p$,   we
have,  for $x_4$ large enough, 
\begin{align*}
\|\pi (0,0,x_4,0)(f_1\otimes \phi ) \|_{M^p(\bR ^2) } &=  \| f_1 \|_{M^p (\bR )} \, \|\cN
                                       _{x_4}\phi\|_{M^p (\bR )} \\
 & = c_p (4+x_4^2)^{\frac{1}{2p}-\frac{1}{4}} \|f_1\|_{M^p} \|\phi\|_{M^p}\\
&\leq C |x_4|^{\frac{1}{p}-\frac{1}{2}} \|f_1\otimes \phi \|_{M^p} \, .   
\end{align*}
Thus  the one-parameter group of (unitary) operators $\pi (0,0,0,x_4,0)$
does not act by isometries on $M^p(\bR ^2)$ for $1\leq p<2$, but is unbounded on
$M^p(\bR ^2)$.  We conclude that  $\colp
(G_{5,3}/Z) \neq M^p(\bR ^2)$. If $p>2$, we use duality to achieve the
same conclusion. 

In fact, using a standard argument we
construct  an 
element $h\in \colp (G_{5,3}/Z)$ that is not in $M^p(\bR ^2)$. Simply choose an
increasing sequence $u_n >0$, such that
\begin{align*}
\frac{1}{n^2} u_n ^{\,\, \frac{1}{p}-\frac{1}{2}} \to \infty \qquad
  \text{ and }   \qquad 
c_p u_n ^{\,\, \frac{1}{p}-\frac{1}{2}} > 2C n^2 \sum _{j=1}^{n-1} \frac{1}{j^2}   u_j^{\,\,
\frac{1}{p}-\frac{1}{2}}   \qquad \forall n \, .
\end{align*}
Fix $f= f_1 \otimes \phi $. Then by absolute convergence and~\eqref{eq:n9b}
$$
h= \sum _{j=1}^\infty \frac{1}{j^2} \pi (0,0,u_j ,0) f \in \colp
(G_{5,3}/Z) \subseteq L^2(\bR ^2) \, ,
$$
but  the sequence of partial sums of $h$  is unbounded in $M^p(\bR ^2)$:
\begin{align*}
  \|\sum _{j=1}^n \frac{1}{j^2} \pi (0,0,u_j ,0) f \|_{M^p} 
&\geq \frac{1}{n^2} \|\pi (0,0,u_j,0)f \|_{M^p} - \sum
 _{j=1}^{n-1}  \frac{1}{j^2} \|\pi (0,0,u_j,0)f \|_{M^p} \\
&\geq  c_p \frac{1}{n^2} u_n ^{\,\, \frac{1}{p}-\frac{1}{2}}- C \, \sum
 _{j=1}^{n-1}  \frac{1}{j^2} u_j  ^{\,\, \frac{1}{p}-\frac{1}{2}} >\frac{1}{2n^2}
  u_n ^{\,\, \frac{1}{p}-\frac{1}{2}}\to \infty \, .
\end{align*}
Thus  $h \not \in M^p(\bR )$. 
\end{proof}

The above argument does not exclude the possibility that $\colp
(G_{5,3}/Z) = M^p_m(\bR ^2)$ for some moderate weight $m$. An
extension of Proposition~\ref{neqs} yields the  general result.

\begin{tm} \label{different}
  Let $1\leq p\leq \infty, p\neq 2$, and $m$ be an arbitrary  moderate weight
  function on $\bR ^4$.  
  Then
  $$
  \colp (G_{5,3}/Z) \neq M^p_m (\bR ^2) \, .
  $$
\end{tm}

\begin{proof}
  Again, we only treat $1\leq p<2$ and use duality for $p>2$. Since
   $\colp (G/Z) \subseteq L^2(\bR ^2)$ by Proposition~\ref{props}(v),
  we only need to compare to those \modsp s $M^p_m(\bR ^2)$ that are
  embedded in $L^2(\bR ^2 )$. Indeed, if $M^p_m(\bR ^2) \not \subseteq
  L^2(\bR ^2)$, then there exists $h\in M^p_m(\bR ^2)
  \setminus L^2(\bR ^2)$ and this function $h$ cannot be in $  \colp
  (G_{5,3}/Z) \subseteq L^2(\bR ^2)$.

  For $p<2$ the embedding
  $M^p_m(\bR ^2)\subseteq L^2(\bR ^2) =  M^2(\bR ^2)$ implies that  $m$ is bounded
  below, $m(\dot x) \geq C_0 >0$ for all $\dot x \in \bR
  ^4$ (check   for instance~\cite[Thm.~12.2.2]{book} and
  \cite[Thm.~4]{fg89mh}). In particular, we also have the embedding  $M^p_m(\bR ^2) \subseteq M^p(\bR ^2)$. 
    We now  use the second proof and compute the norm of $M_u f_1 \otimes
  \cN _u \phi $ in $M^p_m(\bR ^2)$. Since the \stft\
  factors as
  \begin{align*}
  S_{\phi \otimes \phi} (M_u f_1 \otimes \cN _u\phi
  )(x_1,x_2,\xi_1,\xi_2) &= S_\phi (M_u f_1)(x_1,\xi_1) \, S_\phi
  (\cN_u \phi)(x_2,\xi_2)\\    
&= S_\phi f_1(x_1,\xi_1-u) \, S_\phi   (\cN_u \phi)(x_2,\xi_2) \, ,
  \end{align*}
   the norm of $M_u f_1 \otimes
\cN _u\phi $ is
\begin{align*}
  \|M_u f_1 \otimes   \cN _u \phi \|_{M^p_m}^p & =\int _{\bR ^4} 
 |S_\phi f_1(x_1,\xi_1-u)|^p \, |S_\phi   (\cN_u \phi)(x_2,\xi_2)|^p 
 m(x_1,x_2,\xi _1,\xi_2)^p \, dx_1dx_2d\xi_1d\xi_2\\ 
&\geq  C_0^p \int _{\bR ^4} 
|S_\phi f_1(x_1,\xi_1-u)|^p \, |S_\phi   (\cN_u \phi)(x_2,\xi_2)|^p  \, dx_1dx_2d\xi_1d\xi_2\\
&\asymp \int _{\bR ^2} |S_\phi f_1(x_1,\xi_1)|^p \, 
  \|\cN_u \phi \|^p_{M^p} \, dx_1d\xi_1\\
&\asymp \int _{\bR ^2} |S_\phi f_1(x_1,\xi_1)|^p \, 
(1+u^2)^{\frac{1}{2}-\frac{p}{4}}   \, dx_1d\xi_1\\
&= (1+u^2)^{\frac{1}{2}-\frac{p}{4}} \|f_1\|_{M^p}^p 
 \asymp  (1+u^2)^{\frac{1}{2}-\frac{p}{4}} \|f_1\otimes \phi  \|_{M^p}^p \, .
\end{align*}
Again, we have used Proposition~\ref{chirp} and \eqref{eq:m1}. 
We conclude that the orbit  $M_uf_1 \otimes \cN_u\phi $ is bounded
in $\colp (G_{5,3}/Z)$, but unbounded in $M^p_m(\bR ^2)$. Therefore the  two
spaces cannot be equal.   
\end{proof}

The  above arguments are  typical and can be applied to several other groups in
Nielsen's list. 
In  the following  we only deal with unweighted
versions of $\colp $.

\subsection{The group $G_{6,19}$}

This  is the six-dimensional group $G_{6,19} \simeq \bR ^6$
defined by the Lie brackets
$$
[X_6,X_5]=X_4, \quad [X_6,X_3]=X_1, \quad [X_5,X_4]=X_2
$$
with group multiplication
\begin{align*}
\lefteqn{x\cdot y = (x_1, \dots , x_6)\cdot (y_1, \dots , y_6) =} \\ & (x_1+y_1 +
x_6y_3, x_2+y_2+x_5y_4+x_5x_6y_5+\tfrac{1}{2}x_6y_5^2, x_3+y_3, x_4+ y_4+x_6y_5, x_5+ y_5,  x_6+
y_6) \, .
\end{align*}

This group  possesses a two-parameter family of square-integrable
representations $\pi _{\lambda,\mu }, \lambda\mu \neq 0$,  modulo the center $\bR ^2 \times \{ (0,0,0,0)\}$
acting on $L^2(\bR ^2)$ by the operators
\begin{align*}
\lefteqn{\pi _{\lambda ,\mu } (x_1, \dots , x_6) g(s,t) =} \\ & \exp  2\pi i \Big(
\lambda (x_1-x_3  t)   +\mu (x_2-\tfrac{1}{2} x_5^2x_6 -x_4s+x_5x_6s-\tfrac{1}{2} x_6s^2) \Big)
g(s-x_5,t-x_6) \, .
\end{align*}
When using the absolute values $|V^{\pi _{\lambda ,\mu }} _gf|$, the
phase factor $e^{2\pi i (\lambda x_1 +\mu (x_2-x_5^2x_6/2))}$
disappears, and  we may  omit it in the definition of the transform  $V^\pi _g$.
In addition, we identify the quotient $G_{6,19}/Z$ with a convenient  section
$G_{6,19}/Z \to G_{6,19}$ and will write $\dot x = (0,0,x_3, \dots , x_6) \in
G_{6,19}/Z$. Then   the representation can be written   with the help of  \tfs s as 
\begin{equation}
  \label{eq:l8a}
  \pi _{\lambda , \mu }(\dot x) = M_{(\mu (-x_4+x_5x_6), -\lambda
    x_3)} \cN _{\mu x_6}^{(1)} T_{(x_5,x_6)}  =  e^{i\pi \mu x_5^2} M_{(-\mu x_4, -\lambda
    x_3)} T_{(x_5,x_6)} \cN _{\mu x_6}^{(1)}  \, ,
\end{equation}
where  $\cN _{u}^{(1)}g(s,t) = e^{-i\pi u s^2} g(s,t)$ is a multiplication
by a chirp in the variable $s$.   Using the \stft , the cleaned-up \rep\ coefficient is
$$
V_g^{\pi _{\lambda , \mu }}f(\dot x) = S_{\cN^{(1)}_{\mu x_6}g} f \big(
(x_5,x_6), (-\mu x_4,-\lambda x_3)\big) \, .
$$
The underlying phase space consists of the ``frequency/momentum''
variables $x_3,x_4$ and the ``time/position'' variables $x_5,x_6$. The
pairs $(x_4,x_5)$ and $(x_3,x_6)$ are conjugate, because $[X_4,X_5]$
and $[X_3,X_6] \in \mathfrak{z}$. 
Given a weight function $m$ on $G/Z$, the coorbit space with respect 
to $\pi _{\lambda,\mu } $ is defined by the norm
\begin{align*}
\| f \|_{\colpm} ^p &= \int _{\bR ^4} |V^{\pi _{\lambda ,\mu }}_g f
  (\dot x)| ^p \, m(\dot x )^{p} \, d\dot x \\ &= \int_{\bR ^4} |S_{\cN_{\mu x_6}^{(1)}g} f \big(
(x_5,x_6), (-\mu x_4,-\lambda x_3)\big)|^p \, m(\dot x)^p  dx_3 \dots dx_6  \, .  
\end{align*}
By replacing $g(s,t)$ by $g(\mu ^{-1/2}s,t)$ and a change of variables in
the integral, one can  see that the class of coorbit spaces does not depend on the
parameters $\lambda ,\mu \neq 0$ of the representation. Different
parameters affect only the weight function, but not the type of the
space. 

Again the question arises whether we have defined a new class of
function spaces or not. 

  \begin{prop}
   Let $\lambda \mu\neq 0$ and  $p,q\in [1,\infty ], p \neq 2$. Then 

(i) $Co _{\pi_{\lambda,\mu}} L^p (G_{6,19}/Z) \neq M^q(\bR
    ^2)$, and 

(ii) $Co _{\pi_{\lambda,\mu}} L^p(G_{6,19}/Z) \neq \co L^q (G_{5,3}/Z)$. 
  \end{prop}
   
  \begin{proof}
(i) As in the case of $G_{5,3}$ we use tensor products, namely $g(s,t) =
e^{-\pi (s^2+t^2)} = (\phi \otimes \phi )(s,t)$ and $f = \phi \otimes
f_2$ for suitable $f_2$. Then 
\begin{align}
  \label{eq:p1}
  V_g^{\pi_{\lambda,\mu}}f(\dot x)  &= \langle \phi , M_{-\mu
    x_4}T_{x_5} \cN _{\mu x_6}\phi \rangle \, \langle f_2, M_{-\lambda 
     x_3}T_{x_6} \phi\rangle  \notag \\
  &= e^{2\pi i \mu x_4x_5} \overline{S_\phi (\cN _{\mu x_6}\phi )(-x_5,
  \mu x_4)} \, S_\phi f_2( x_6,-\lambda x_3) \, .
\end{align}
Consequently,
\begin{align}
  \|f\|_{Co _{\pi_{\lambda,\mu} } L^p} ^p &= \int _{\bR ^4}    |V_g^{\pi_{\lambda,\mu}}f(\dot x) |^p \,
                      dx_3dx_4dx_5dx_6 \notag \\ 
& = \int _{\bR ^2} \|\cN _{\mu x_6}\phi \|_{M^p(\bR )}^p \,  |S_\phi
f_2( x_6,-\lambda x_3)|^p \, dx_3dx_6\, . 
\end{align}
Since by~\eqref{eq:a3} we have $\|\cN _{\mu x_6} \phi \|_{M^p(\bR )} \asymp
  (4+\mu^2 x_6^2)^{\frac{1}{2p} - \frac{1}{4}} \asymp
  (1+x_6^2)^{\frac{1}{2p} - \frac{1}{4}}:= v(x_6,x_3)$, the $Co _{\pi_{\lambda,\mu} } L^p (G_{6,19}/Z)$-norm of $\phi
  \otimes f_2$ is 
  $$
  \|\phi \otimes f_2\|_{Co _{\pi_{\lambda,\mu} } L^p } ^p \asymp \int _{\bR ^2} |S_\phi
  f_2(x_6,-\lambda x_3)|^p v(x_6,-\lambda x_3)^p \, dx_6 dx_3 \, .
  $$
 We see that $\phi \otimes f_2$ is in $Co _{\pi_{\lambda,\mu} } L^p(G_{6,19}/Z) $, \fif\ $f_2 \in
  M^p_v(\bR )$. By choosing $f_2\in M^p(\bR ) \setminus M^p_v(\bR )$, we have
  constructed a function $\phi \otimes f_2 \in M^p(\bR ^2)$, but $\phi \otimes
  f_2 \not \in Co _{\pi_{\lambda,\mu} } L^p(G_{6,19}/Z) $.

  (ii)  We take the one-parameter subgroup
    $\pi (0,0,0,u,0)$ of $G_{5,3}$ acting on $ f=  f_1 \otimes f_2 $
    and show that it acts unboundedly on $Co _{\pi _{\lambda,\mu } } L^p
    (G_{6,19}/Z)$. 
 By \eqref{eq:l3},  
$\pi (0,0,0,u,0)(f_1\otimes f_2)  = M_uf_1 \otimes \cN _uf_2\phi $ and 
$\pi _{\lambda,\mu}(0,0,x_3, \dots , x_6)(\phi \otimes \phi )  = M_{-x_4}T_{x_5}\cN _{x_6}\phi
\otimes M_{-x_3}T_{x_6}\phi $. The corresponding representation
coefficient is then 
\begin{align*}
  V_{\phi \otimes \phi}^{\pi _{\lambda,\mu}}(\pi (0,0,0,u,0)(f_1
  \otimes f_2)) (x_3,
  \dots, x_6) &= \langle M_uf_1 \otimes \cN _uf_2, M_{-\mu
                x_4}T_{x_5}\cN _{\mu x_6}\phi
\otimes M_{-\lambda x_3}T_{x_6}\phi \rangle \\
&= \langle T_{-x_5} M_{u+\mu x_4} f_1, \cN _{\mu x_6}\phi \rangle \, \langle
 \cN _u f_2,  M_{-\lambda x_3}T_{x_6}\phi \rangle \\
&= e^{2\pi i  x_5(u+\mu x_4)} \overline{S_{f_1} (\cN _{\mu x_6}\phi) (-x_5,
 \mu  x_4+u) } \, S_\phi (\cN _u f_2)(x_6,-\lambda x_3) \, .
\end{align*}
We now choose $f_1= f_2 = \phi $ and obtain
\begin{align*}
\lefteqn{  \|\pi (0,0,0,u,0)(\phi \otimes \phi )\|_{Co _{\pi _{\lambda,\mu}} L^p  }^p=}\\ &= \int_{\bR ^2}\int_{\bR ^2}
|S_{\phi} (\cN _{\mu x_6}\phi) (-x_5,   \mu x_4+u)|^p  \, |S_\phi (\cN   _u\phi)(x_6,-\lambda x_3)|^p
\, dx_4dx_5 \, dx_3dx_6 \\
&= \mu \inv \int_{\bR ^2} \|\cN _{x_6}\phi \|_{M^p}^p |S_\phi (\cN _u \phi)(x_6,-\lambda x_3)|^p
\, dx_3dx_6 \\ 
&\asymp \int_{\bR ^2} (1+|x_6^2|)^{\frac{1}{2}-\frac{p}{4}}  |S_\phi
                  (\cN _u\phi )(x_6,-\lambda x_3)|^p
\, dx_3dx_6  \, .
\end{align*}
In the last expression we substitute \eqref{eq:app0} for $|S_\phi (\cN _u\phi )|$ and
continue with
\begin{align*}
  \lefteqn{  \|\pi (0,0,0,u,0)(\phi \otimes \phi )\|_{Co _{\pi
  _{\lambda,\mu}} L^p  }^p\asymp }\\
&\asymp \int_{\bR ^2} (1+|x_6^2|)^{\frac{1}{2}-\frac{p}{4}} (4+u^2)^{-p/4}
                 e^{-\pi p x_6^2/2} e^{-2\pi p(-\lambda x_3+ux_6/2)^2/(4+u^2)}
\, dx_3dx_6 \\
&= \lambda ^{-1/2}  \int_{\bR } (1+|x_6^2|)^{\frac{1}{2}-\frac{p}{4}} (4+u^2)^{-\frac{p}{4}+\frac{1}{2}}
 \, e^{-\pi p x_6^2/2} \, dx_6\\
&\asymp   (4+u^2)^{-\frac{p}{4}+\frac{1}{2}} \, .
\end{align*}
 We conclude that
$$
  \|\pi (0,0,0,u,0)(\phi \otimes \phi )\|_{Co _{\pi_{\lambda,\mu}} L^p (G_{6,19}/Z) } \asymp
  (1+u^2)^{\frac{1}{2p}-\frac{1}{4}} \, .
 $$
Since $\pi (0,0,0,u,0)$ is an isometry on $\colp (G_{5,3}/Z)$, but not
on $Co _{\pi _{\lambda,\mu}} L^p(G_{6,19}/Z)$, these two spaces cannot be equal. 
\end{proof}

To summarize, we have constructed three families of function spaces on
$\bR ^2$ that are obtained as coorbit spaces with respect to the 
 nilpotent Lie groups $\bH _2,
G_{5,3}$ and $G_{6,19}$. After quotienting out their  centers,
these  groups are non-isomorphic. As  these three families of
coorbit spaces have different invariance properties, we were able to
show that they  are
distinct. From the point of view of the theory of function spaces, we
have discovered two brand-new families of function spaces. 

\subsection{The Dynin-Folland group} 
Let $\mathfrak{g} := \bR -  \mathrm{span}
\, \{Z, Y_1 , Y_2 , Y_3 , X_1 , X_2 , X_3 \}$ with
non-trivial Lie brackets
$$
[ X_3, Y_1 ] = [X_2 , Y_2 ] = [ X_1, Y_3 ] = Z, [ X_2, Y_3 ] =
\tfrac{1}{2} Y_1 ,
[ X_3, Y_3 ] = -\tfrac{1}{2} Y_2 \,\, \text{ and } \,\,  [X_3 , X_2 ] = X_1  \, .
$$
We label the elements of $G$ as
$(z,y_1,y_2,y_3,x_1,x_2,x_3)$.
This group possesses a one-parameter family of  square-integrable
representations modulo center $\pi _\lambda \in SI(G/Z)$ that act on
$L^2(\bR ^3)$, or more precisely, on $L^2(\bH _1)$ of the Heisenberg group
$\bH _1$.  As they are
obtained by a dilation  from each other, we may restrict to
$\lambda =1$   and set
\begin{equation}
  \label{eq:n6}
 \Big(\pi(z,y_1,y_2,y_3,x_1,x_2,x_3)f \Big)(t_3 , t_2 , t_1 ) =
 e^{2\pi i(z+\sum _{k=1}^3 t_j y_j - t_1 t_2 y_3/2)}
 f (t_3 + x_1 + t_1 x_2 , t_2 + x_2 , t_1 + x_3 ) \, .
\end{equation}
Again,  this  representation acts on $f\in L^2(\bR ^3 )$ by
means of generalized modulations and translations as in \eqref{eq:1}. In fact, the
translations are just the multiplication in the Heisenberg group
$\bH_1$, the modulations are just the standard modulations.  The
most interesting item is the multiplication by the chirp
$f(t_3,t_2,t_1)  \to  e^{-\pi i t_1t_2 u}f(t_3,t_2,t_1)$. 
For background and detailed derivations of the multiplication and the
representations of $G$ we refer to  \cite{FRR19,rottensteinerthesis}.

We now show that the resulting coorbit spaces are different from the
\modsp s. This was  one of the main results
of \cite{FRR19}, where it was proved in much greater generality.
The proof in \cite{FRR19} requires substantial
parts of the theory of decomposition spaces~\cite{FV15,Voi18} and is based on the subtle identification of
$\colp $ with a decomposition space.
It seems therefore worthwhile to have a simple,
short, and self-contained   proof. Ours  is based on the different invariance properties of the coorbit
spaces.

\begin{prop}
  For $p\in [1,\infty], p\neq 2$ we have
  $\colp (G/Z)  \neq M^p(\bR ^3)$. 
\end{prop}

\begin{proof}
  The argument is similar to the proof of Proposition~\ref{neqs}. By
  inspection of \eqref{eq:n6} we see that the one-parameter subgroup
  $(0,0,0,u,0,0,0) = e^{uY_3}$ acts on $f\in L^2(\bR )$ as
  \begin{equation}
    \label{eq:n10}
    \pi (e^{uY_3})f(t_3,t_2,t_1) = e^{2\pi i ut_3} \, e^{-\pi i u
      t_1t_2} f(t_3,t_2,t_1) \, .    
  \end{equation}
  We recognize the first exponential as a modulation with respect to
  the variable $t_3$ and the second exponential as a chirp 
  with respect to the variables $t_2,t_1$. To simplify, we write
  $\bar{t} = (t_2,t_1)$ and the quadratic form $(t_1,t_2) \to t_1t_2$
  as $C\bar{t} \cdot \bar{t}$ with the matrix $C = \tfrac{1}{2}
  \Big(\begin{smallmatrix}
    0 & 1 \\ 1 & 0
  \end{smallmatrix} \Big) $. Again we take $f$ to be a tensor product
  $f (t_3,t_2,t_1)  = f_3(t_3) \phi _2 (\bar{t})$ with the
  two-dimensional Gaussian $\phi _2 (\bar{t}) = e^{-\pi \bar{t} \cdot
    \bar t}$. Then the action of the one-parameter group generated by
  $Y_3$ is 
  \begin{equation}
    \label{eq:n11}
    \pi (e^{uY3}) (f_3 \otimes \phi ) = M_u f_3 \otimes  \cN _{uC}
    \phi _2 \, , 
  \end{equation}
  and its \stft\ is 
  \begin{equation}
    \label{eq:n12}
    S_{\phi \otimes \phi _2} \big( \pi (e^{uY_3}) (f_3 \otimes \phi
    _2) \big) (x_3,\bar x, \xi
    _3, \bar \xi ) = S_\phi (M_uf_3)(x_3,\xi _3) \, S_{\phi _2} (\cN
    _{uC} \phi _2) (\bar x, \bar \xi )\, .
  \end{equation}
Using the general version of Proposition~\ref{chirp}, specifically
\eqref{eq:m2},  the \modsp -norm
on $\bR ^3$ is therefore  
\begin{align*}
  \|    S_{\phi \otimes \phi _2} \big( \pi (e^{uY_3}) (f_3 \otimes \phi
  _2) \big) \|_{M^p}& = \|M_u f_3\|_{M^p( \bR )} \, \|\cN _{uC} \phi _2\|_{M^p(\bR ^2)}\\ 
&\asymp \|f_3\|_{M^p( \bR )} \,  (4 + \tfrac{u^2}{4})^{\tfrac{1}{p}-\tfrac{1}{2}}
\, .                      
\end{align*}
 For $1\leq p <2$ the one-parameter group $e^{uY_3}$ acts unboundedly on
$M^p(\bR ^3)$, and we conclude that $\colp \neq M^p(\bR ^3) $ for
$1\leq p<2$. For $p>2$ we use duality.  
\end{proof}

\section{Atomic decompositions}

The abstract version of atomic decompositions in coorbit spaces was
developed in~\cite{fg89jfa,fg89mh,gro91}. They yield  series
expansions of elements in a coorbit space with respect to elements in
the orbit of the representation. Following a costum in coorbit theory,
we briefly summarize
 these results and update them with recent results that are  specific
 for  nilpotent groups.  

The first goal is  to construct a
  ``window'' function $g\in \cH $ and 
  (relatively) discrete set $\Lambda \subseteq G$, such that there
  exist constants $A,B>0$, such that
  \begin{equation}
    \label{eq:3}
A \|f\|_\cH ^2 \leq     \sum _{\lambda \in \Lambda } |\langle f, \pi (\lambda )g \rangle |^2
\leq B \|f\|_\cH ^2 \qquad \forall f \in \cH \, .
\end{equation}
A set $\cG (g,\Lambda ) = \{ \pi (\lambda )g: \lambda \in \Lambda \}$
satisfying \eqref{eq:3} is called a (coherent) frame.

If $G$ is nilpotent with center $Z$ and $\pi $ is irreducible, then,
for arbitrary $z_\lambda \in Z$,  the set $\{ \pi (z_\lambda
\lambda )g : \lambda \in \Lambda  \}$ is also a frame with the same
 bounds $A,B$. We could have started with a projective representation
of $G/Z$ and formulated everything in $G/Z$. However, since the
construction of irreducible unitary representations of a nilpotent Lie
group is on the level of $G$, we prefer to work on the group level.
To take into account this ambiguity, we will use  discrete sets
$\tilde{\Lambda}\subseteq G/Z$ and an arbitrary section
$\tilde{\Lambda} \to \Lambda \subseteq G$.

We also need the notion of ``nice'' vectors. We say that $g\in \cH $
is nice, if $g = \pi (k)g_0= \int _G k(x) \pi (x) g_0 \, dx$ for some
$g_0 \in \cH $ and $k$ continuous with compact support (or $k$ in the
amalgam space $W(C,\ell ^1)(G)$). 

\begin{tm}[Existence of coherent frames]  \label{tm1-abstr}
Let $(\pi , \cH )$ be an irreducible unitary representation of $G$
that is square-integrable modulo center   and let $g\in \cH $ be a
nice vector. Then there exists a neighborhood $U\subseteq G/Z$ of $e$
with the following property:

Assume that  $\tilde \Lambda \subseteq G/Z $ is $U$-dense and
relatively separated, i.e., $\bigcup _{\tilde{\lambda} \in\tilde{\Lambda }}
\tilde \lambda U = G/Z$ and $\max _{x\in G/Z} \# (\tilde{\Lambda} \cap xU)
<\infty $, and  let $\Lambda \subseteq G $ be some  preimage of
$\tilde{\Lambda }$.

Then the set $\cG (g,\Lambda ) = \{ \pi (\lambda ) g :
\lambda \in \Lambda \}$ is a frame for $\cH $. 
\end{tm}
 The existence of coherent frames was first proved in
~\cite[Thm.~4.1]{FG07} for arbitrary square-integrable
representations. Note that the original coorbit theory~\cite{fg89jfa}
required somewhat  stronger assumptions on $g$ and $\pi $, namely the 
integrability of $\pi $ and $V_gg \in W(C,\ell ^1)(G)$. 
For nilpotent Lie groups the square-integrability of $\pi $  modulo the center
automatically implies automatically its
integrability~\cite{CG90}. Theorem~\ref{tm1-abstr}  
also follows from the recent work~\cite{FS19} on the discretization of
arbitrary reproducing kernel Hilbert spaces. 

Theorem~\ref{tm1-abstr} amounts to a non-uniform sampling theorem for
the  transform  $V^\pi _gf$. For nilpotent Lie groups one can choose
the set $\Lambda $ to be a lattice or a quasi-lattice. We use the
following terminology: a set $\Lambda \subseteq G$ is called a
quasi-lattice of $G$ with relatively compact fundamental domain $K$, if $G = \bigcup
_{\lambda \in \Lambda } \lambda K = G$ and $\lambda K \cap \lambda ' K
= \emptyset$ for $\lambda \neq \lambda '$. If, in addition,  $\Lambda $ is a
subgroup, then $\Lambda $ is called a lattice of $G$. (Note subtle
differences of  terminology in the literature!). Quasi-lattices
exist in every simply connected nilpotent Lie group~\cite{FG07},
whereas  the existence of a lattice requires  a rational structure of
the Lie algebra of $G$~\cite{CG90}. 

For nilpotent groups we can add more structure in  Theorem~\ref{tm1-abstr}. 

\begin{prop}
Let $(\pi , \cH )$ be an irreducible unitary representation of $G$
that is square-integrable modulo center   and let $g\in \cH $ be a
nice vector. Then there exists a  quasi-lattice $\tilde{\Lambda} \subseteq G/Z$, such that $\cG (g,\Lambda ) = \{ \pi (\lambda ) g :
\lambda \in \Lambda \}$ is a frame for $\cH $. 
\end{prop}

\begin{proof}
  In view of Theorem~\ref{tm1-abstr} we only  need the existence of
  sufficiently fine quasi-lattices. This is essentially proved in
  \cite[Lemma~3.9]{GR18}. In brief, fix a strong Malcev basis
  $X_1, \dots X_n $ of $\mathfrak{g}/\mathfrak{z}$. Given a
  neighborhood $U\subseteq G/Z$, choose $\epsilon >0$, such that $K=
  \{ e^{t_1X_1} \dots e^{t_nX_n}: -\epsilon /2 \leq t_j < \epsilon
  /2\} $ is contained in $U$. Then the set $\Gamma = \{ e^{k_n\epsilon
    X_n} \dots e^{k_1\epsilon X_1} : k_n \in \bZ \}$ is a quasi-lattice
  with fundamental domain $K$. The proof by induction is identical to
  ~\cite{GR18}. 
\end{proof}

This qualitative result can be complemented by a necessary density
condition in the style of Landau~\cite{landau67}.  
As an appropriate metric on $G/Z$ we choose a word metric: fix a
symmetric neighborhood $W=W\inv $ of $e$ in $G/Z$ and let $d(x,y) = \min
\{ n\in \bN : y\inv x \in W^n\}$ for $x\neq y$. Denoting the balls with
respect to this metric by $\bar B(x,r) = \{ y\in G/Z: d(x,y) \leq
r\}=x\bar B(e,r)$,
the lower Beurling density of a set $\Lambda \subseteq G/Z $ is given by
$$
D^-(\tilde{\Lambda } ) = \liminf _{r\to \infty } \inf _{x\in G/Z} \frac{\#
\tilde{\Lambda} \cap \bar B(x,r)}{|\bar B(x,r)|} \, .
$$
As in $\rd $ the density $D^-(\tilde{\Lambda })$ is the average number of
points in a ball of radius $1$. 
For coherent frames with
respect to a square-integrable irreducible \rep\ the following density
result was proved in ~\cite{FGHKR}. 
\begin{tm}[Necessary density condition] \label{tm:density}
  Let $G$ be a nilpotent Lie group and $\pi \in SI(G/Z)$ be a
  square-integrable representation of $G$ modulo the center with
  formal dimension $d_\pi $. Let $g\in \cH $ be a nice vector  and
  $\tilde{\Lambda } \subseteq
  G/Z$.
  
  If $\cG (g,\Lambda ) $ is a frame for $\cH $, then
  $D^-(\tilde{\Lambda } )
  \geq d_\pi $
\end{tm}

For certain nilpotent Lie group one can even prove the  existence of
orthonormal bases in the orbit of every square-integrable
representation modulo the center~\cite{GR18}.
\begin{tm} \label{onb}
  Let $G$ be a graded Lie group   with  one-dimensional  center  and
$(\pi , \cH )$ be  a
square-integrable irreducible unitary representation modulo
center. Then there exists a discrete set $\Lambda \subseteq G$ and a
function $g\in \cH$, such
that $\cG (g,\Lambda )$ is an orthonormal basis for $\cH $. 
\end{tm}

 See also~\cite{Ous18,Ous19} for additional
observations. It is currently an open problem whether Theorem~\ref{onb} can be
extended to all nilpotent Lie groups and all square-integrable
representations modulo the center.

Theorem~\ref{tm1-abstr} possesses a version for  coorbit spaces in~\cite{fg89jfa} as
follows:
\begin{tm}[Banach frames] \label{tm:bf}
Let $g$ be a nice vector in $ \co L^1_v(G)$, e.g., $g\in \cS
(G/Z)$. Assume that $1\leq p \leq \infty $ and that $m$ is
$v$-moderate. Then   there exists a neighborhood U in $G/Z$ with the
following 
property:
If   $\tilde{\Lambda } \subseteq G/Z$ is $U$-dense and relatively
separated and $\Lambda \subseteq G$ is a section of $\tilde{\Lambda}$, then 
there exists a dual frame
$\{ e_\lambda : \lambda \in \Lambda \}$ in  $\co
L^1_v(G/Z)$, such that for all $f\in \colpm (G/Z) $ 
\begin{align}
f& = \sum _{\lambda \in \Lambda } \langle f, e_\lambda \rangle \pi
   (\lambda )g = \sum _{\lambda \in \Lambda } \langle f, \pi (\lambda
   ) g \rangle e_\lambda \label{eq:6} \\ 
  \| \big( \langle f, \pi (\lambda ) g\rangle \big)_{\lambda \in
    \Lambda }  \|_{\lpm} & \asymp   \| \big( \langle f, e_\lambda \rangle \big)_{\lambda \in
    \Lambda }  \|_{\lpm} \asymp   \|f\|_{\colpm} \, .  \label{eq:5}
\end{align}
  The convergence in \eqref{eq:6} is in the norm of $\colpm $ for $p<
  \infty $ and in $\sigma (\co L^\infty _{1/v}, \co L^1_v)$ for
  $p=\infty $.
\end{tm}

  \begin{rem}
1.   If $\Lambda $ is  a lattice in $G/Z$, then there
  exists a dual window $\gamma \in \co L^1_v$, such that the dual 
  frame is given by $\{ \pi (\lambda ) \gamma : \lambda \in \Lambda
  \}$. The proof is the same as for the Schr\"odinger representation
  of the Heisenberg group. 

2. More information about the dual $\{e_\lambda \}$ is derived in
~\cite{FRRV20,RVV20}. 
    \end{rem}

    \section*{Appendix}

    For quick reference we provide the computation of the \stft\ of
    chirps, as it can be found, for instance
    in~\cite{CN08a}. Let $B,C$ two \emph{real}-valued symmetric
    $d\times d$-matrices with $B$ positive semi-definite and $A =
    B+iC$.  Then $A^T = A$ as well. The
    associated Gaussian is $\phi _A(x) = e^{-\pi Ax\cdot x}$, where
    $x\cdot y = \sum _j x_j y_j$ is the inner product on $\rd $. The
    Fourier transform of $\phi _A$ is given as 
    \begin{equation}
      \label{eq:app1}
      \widehat{\phi _A}(\xi ) = (\det A) ^{-1/2} e^{-\pi A\inv \xi
        \cdot \xi } \, , \qquad \xi \in \rd \, .
    \end{equation}
This formula holds for $\xi \in \rd $ and  real-valued
positive-definite  $A$. By
analytic continuation \eqref{eq:app1}  extends to $\xi \in \cd$ and
complex-valued matrices with the understanding that the branch of the
square-root is determined by the requirement that $(\det A)^{-1/2}>0$
for real-valued positive-definite $A$. See~\cite{folland89}.

Now let $\cN _C f(t) = e^{-\pi i C t\cdot t} f(t), t\in \rd $,  be the
operator of multiplication by the chirp $e^{-\pi i C t\cdot t}$
(with $C^T=C$). We compute the modulus of the \stft\ with respect to the Gaussian
$\phi = \phi _\mathrm{I}$. 
\begin{align}
  S_\phi (\cN _C \phi ) (x,\xi ) &= \intrd e^{-\pi i Ct\cdot t}
 e^{-\pi |t|^2} e^{-\pi |t-x|^2} \, e^{-2\pi i \xi \cdot t} \, dt
                                   \notag \\
&= e^{-\pi |x|^2} \,  \intrd e^{-\pi (2\id + iC)t\cdot t} e^{-2\pi i
                                               t\cdot (\xi +ix)} \,
                                               dt \notag \\
&= \det (2\id + iC) ^{-1/2} \, e^{-\pi |x|^2} \,  e^{-\pi (2\id (\xi
 +ix)\cdot  (\xi +ix)}  \, .  \label{eq:ll1}
\end{align}
The determinant is  
$$
|\det (2\id +iC)| = |\det (2\id +iC) \det (2\id - iC)|^{1/2} = \det
(4\id + C^2) ^{1/2} \, .
$$
Writing $D= (4\id + C^2)\inv $ and $(2\id + iC)\inv = (4\id + C^2)\inv
(2\id - iC) = 2D - iDC$, we find after some algebraic manipulations
that the real part of the exponent in~\eqref{eq:ll1}  is given by
$$
\mathrm{Re} \, D(2\id -iC) (\xi +ix) \cdot (\xi +ix) + x\cdot x  = 2D\xi \cdot \xi
+(\emph{I} -  2D) x\cdot x + 2DC\xi \cdot x \, .
$$
This is again a quadratic form, now on $\bR ^{4d}$, and for its
description we use the following abbreviations: $z=(\xi,x)\in \bR
^{4d}$ and
$$
\Delta =
\begin{pmatrix}
  2D &  DC \\ DC & \emph{I}-2D 
\end{pmatrix} \, .
$$
With this notation  the modulus of the \stft\ is
\begin{equation}
  \label{eq:app2b}
  |S_\phi (\cN _C\phi )(x,\xi )| = \det (4\id + C^2)^{-1/4} e^{-\pi
    \Delta z\cdot z} \,  ,
\end{equation}
and its $p$-norm is therefore
\begin{align}
  \|\cN _C\phi \|_{M^p} &= \|S_\phi (\cN _C \phi )\|_{L^p(\bR ^{4d})}
                          \notag \\
  &= \det (4\id + C^2)^{-1/4} \, (\det p\Delta )^{-\frac{1}{2p}} \, ,   \label{eq:app3b}
\end{align}
where the last identity is obtained from \eqref{eq:app1} applied to
the Gaussian $e^{-\pi p 
  \Delta z\cdot z}$ at $z=0$. To compute the determinant of $\Delta $
we use the (block) factorization 
$$
\Delta
\begin{pmatrix}
  0 & C \\ \emph{I} &-2\mathrm{I}
\end{pmatrix}
=
\begin{pmatrix}
  DC & 0 \\ \id -2D & -\id
\end{pmatrix}
$$
As the determinant of a matrix with a $0$-block is easy to compute, we
obtain
$$
\det \Delta \, \det C \det (-\id ) = \det (DC) \, \det (-\id ) \, ,
$$
and this implies that
\begin{equation}
  \label{eq:p2}
  \det \Delta = \det D = \det (4\id + C^2)\inv \, .
\end{equation}
This derivation is rigorous for invertible $C$ and extends to
arbitrary $C$ by continuity. See~\cite{folland89}, Appendix. 
The final result is therefore 
\begin{equation}
  \label{eq:p3}
\|S_\phi (\cN _C \phi )\|_{L^p(\bR ^{2d})} = (p^{-\frac{1}{2p}})^{4d} \det
(4\id + C^2)^{\frac{1}{2p} - \frac{1}{4}} \, .
\end{equation}

\vspace{3mm}

\textbf{Acknowledgement:} I would like to thank Jordi van Velthoven,
Univ.~of Vienna, and David Rottensteiner, Ghent University, for
their useful feedback. 

\def\cprime{$'$} \def\cprime{$'$} \def\cprime{$'$} \def\cprime{$'$}
  \def\cprime{$'$} \def\cprime{$'$}


\end{document}